\documentclass[12pt]{article}
\usepackage{amsmath, amsfonts, amssymb, amsthm,mathrsfs}
\usepackage[colorlinks = true, linkcolor = blue]{hyperref}
\usepackage{tikz-cd}
\newtheorem{theorem}{Theorem}[section]
\newtheorem{question}{Question}
\newtheorem{proposition}[theorem]{Proposition}

\newtheorem{corollary}[theorem]{Corollary}
\newtheorem{lemma}[theorem]{Lemma}

\theoremstyle{remark}
\newtheorem{remark}[theorem]{\bf Remark}

\mathchardef\mhyphen="2D

\newcommand{\inv}{{}^{-1}}
\newcommand{\isomor}{\,\raisebox{4pt}{$\sim$}{\kern -.89em\to}\,}

\newcommand{\PP}{\mathbb P}

\newcommand{\hh}{\mathfrak h}

\newcommand{\frg}{\mathfrak g}

\newcommand{\bb}{\mathfrak b}

\newcommand{\CC}{\mathbb C}
\newcommand{\RR}{\mathbb R}
\newcommand{\ZZ}{\mathbb Z}

\begin{document}
	
		\title{On universal subspaces for Lie groups}
	\author{Saurav Bhaumik and Arunava Mandal}

\maketitle
	
	\begin{abstract}
		 Let $U$ be a finite dimentional vector space over $\mathbb R$ or $\mathbb C$, and let $\rho:G\to GL(U)$
		 be a representation of a connected Lie group $G$. A linear subspace $V\subset U$ is called universal if every orbit of $G$ meets $V$. We study universal subspaces for Lie groups, especially compact Lie groups. 
		 Jinpeng and Dokovi\'{c} approached universality for compact groups through a certain topological obstruction. They showed that the non-vanishing of the obstruction class is sufficient for the universality of $V$, and asked whether it is also necessary under certain conditions. 
		 In this article, we show that the answer to the question is negative in general, but we discuss some important situations where the answer is positive. We show that if $G$ is solvable and $\rho:G\to GL(U)$ is a complex representation, then the only universal complex subspace is $U$ itself. We also investigate the question of universality for a Levi subgroup.
                 \end{abstract}
	

 {\it Subjclass} [2010] {Primary: 22E46, 57R20, Secondary: 22C05, 57T15}

\section{Introduction}

Let $G$ be a connected Lie group and let $\rho:G\to GL(U)$
 be a finite dimensional (real or complex) representation of $G$. A linear subspace $V\subset U$ is called \emph{universal in $U$} (Jinpeng and Dokovi\'{c} \cite{A-D1}) if for every $u\in U$, the $G$ orbit through $u$ intersects $V$ i.e. $\rho(g)(u)\in V$ for some $g\in G$.
A classical example is Schur's triangularization theorem, which says that the complex vector space of upper triangular matrices is universal for the complexified adjoint representation of $U(n)$. For the complexified adjoint representation of a semisimple compact connected Lie group $G$ on its complexified Lie algebra $\frg\otimes_{\mathbb R} \CC$ (where $\frg$ is a Lie algebra of $G$), the sum of root spaces is universal. An important class of examples for universal subspace arises from the work of Jorge Galindo, Pierre de la Harpe, and Thierry Vust \cite{G-H-T}, who investigated how the irreducibility of a representation of connected complex Lie group forces certain other geometrical properties of the orbits. They showed that any complex hyperplane in a finite-dimensional complex irreducible representation of a reductive connected complex Lie group is universal. Gichev \cite{G} showed that the same is true for compact connected Lie groups.  For more examples of universal subspace, we refer to \cite{A-D1}. The availability of a large number of examples for `universal subspace' and its applications in matrix theory (see \cite{A-D1}, \cite{A-D2}, \cite{D-T}, and the references cited therein) is one of the primary motivations to study this topic.

One of the main themes in \cite{A-D1} is to find necessary and sufficient conditions for a linear subspace to be universal, in terms of  certain topological obstructions. They proved that the nonvanishing of that obstruction class is sufficient for the universality of the subspace in general. They posed a question (on p.19 of \cite{A-D1}) whether the nonvanishing of the obstruction class in a certain specific situation is necessary for the subspace to be universal. In our note we provide an example (rather a class of examples) to show that the answer is negative in general. However, we discuss some important situations where the answer is indeed positive. A key observation is that if $G$ is a connected Lie group and if $\hh\subset \frg$ is a Lie subalgebra that is universal for the adjoint representation, then $\hh$ is the Lie algebra of a \emph{closed} subgroup $H\subset G$ (Proposition \ref{universal-closed}). Apart from this, we investigate universality for possibly noncompact solvable Lie groups and Levi subgroups of connected Lie groups.

We recall from \cite{A-D1} the construction of the topological obstruction. Given a representation $\rho:G\to GL(U)$ and a linear subspace $V\subset U$, let $G_V$ be the closed subgroup of $g\in G$ such that $\rho(g) (V)\subset V$. Consider $W=U/V$. Then $U$, $V$ and $W$ are representations of $G_V$. Since $G\to G/G_V=:M$ is a principal $G_V$-bundle, there are the three associated $G$-equivariant vector bundles $E_U,E_V$ and $E_W$ associated to $U,V$ and $W$, respectively. There is a $G$-equivariant isomorphism of vector bundles $E_U\isomor M\times U: [(g,u)]\mapsto ([g],gu)$. Every vector $u\in U$, viewed as a constant section of $E_U$, induces a section $s_u$ of $E_W=E_U/E_V$. Then $V$ is universal if and only if for every $u\in U$, the constant section $s_u$ has a zero (cf. \cite{A-D1}). If $E_W$ is orientable, then it has the Euler class $e(E_W)$. If this Euler class is nonzero, there is no nowhere vanishing section of $E_W$, so $V$ is universal. In case $U$ is a complex representation and $V$ is a complex linear subspace, $E_W$ is a complex vector bundle, and its top Chern class is the same as the Euler class of the underlying real vector bundle.

Let $T$ be a maximal torus of $G$. In case $V$ is $T$-invariant or equivalently $T\subset G_V$, consider the vector bundle $E_W$ over $G/T$. If $V$ is universal, then ${\rm dim}_\RR V\geq {\rm dim}_\RR U - {\rm dim}_\RR G/T$. A $T$-invariant subspace $V$ is said to have \emph{optimal dimension} if ${\rm dim}_\RR V= {\rm dim}_\RR U - {\rm dim}_\RR G/T$. In this case, the Euler class of $E_W$ over $G/T$ lives in the top dimensional cohomology, hence gives rise to a number $C_V=\langle e(E_W),[G/T]\rangle$ via Poincar\'e duality, where $[G/T]$ is the fundamental class of $G/T$. If $C_V$ is nonzero then $V$ is universal in $U$ (Theorem 4.2 of \cite{A-D1}). Since there is a description of the cohomology $H^*(G/T,\ZZ)$, in many cases $C_V$ can be calculated and shown to be nonzero, hence in those cases $V$ is universal (\S 5 of \cite{A-D1}). The converse holds only  under additional condition (Theorem 4.4 of \cite{A-D1}) and fails in general (\S 6 of \cite{A-D1}).

 If $V$ is universal, then ${\rm dim}_\RR V\geq {\rm dim}_\RR U - {\rm dim}_\RR G/G_V$, where only real dimensions are taken (Lemma 4.1 of \cite{A-D1}). Now suppose \begin{equation}\label{dimension}{\rm rank}_\RR (W)={\rm dim}_\RR(G/G_V)=2r.\end{equation} Suppose $E_W$ is orientable. Then $E_W$ has a nowhere vanishing section if and only if the Euler class $e(E_W)\in H^{2r}(G/G_V,\ZZ)$ is zero (Corollary 14.4 of \cite{B}). Jinpeng and Dokovi\'{c} considered an example where $V$ is universal, the Euler class of $E_W$ vanishes over $G/T$ but not over $G/G_V$. They asked the following question in \cite{A-D1}.

	\begin{question}\label{question} Suppose ${\rm rank}_\RR(W)={\rm dim}(M)$ and $T\subset G_V$. Does the universality of $V$ imply that certain obstruction class of $E_W$ over $G/G_V$, for the existence of nowhere vanishing sections, is nonzero?
\end{question}

The two conditions on the rank and $T$-invariance of $V$ are necessary, as they show by examples where the universality of $V$ does not imply the vanishing of certain obstruction class of $E_W$ if either of the above conditions is violated (cf. \S 6 of \cite{A-D1}).
In this context, we investigate the topic further and explore the relationship between certain topological obstructions and universality.

We show that in general, the answer to the above question is negative (see \S 3). When the rank of the vector bundle is equal to the dimension of the base, the Euler class is the final obstruction to the existence of a nowhere vanishing section. We give an example where $V$ is universal, the two conditions in the question hold, yet the Euler class of $E_W$ over $G/G_V$ vanishes.

However, we show that for the following three important classes of examples the universality of $V$ is equivalent to the nonvanishing of the Euler class.

If $G$ is a complex connected Lie group, and $\rho:G\to GL(U)$ is a holomorphic representation of $G$ such that $G/G_V$ is compact, then under the assumption (\ref{dimension}), $V$ is universal if and only if the top Chern class of $E_W$ is nonzero (Theorem \ref{holomorphic-case-1}).

Let $H$ be a closed Lie subgroup of a compact connected Lie group $G$, and we denote the Lie algebra of $H$ by $\hh$.
Consider $U=\frg$, $V=\hh$, and $W=\frg/\hh$. Then the bundle $E_W$ is the tangent bundle $T(G/H)$. We prove using a result of Hopf and Samelson (cf. \cite{H-S}) that the universality of $\hh$ is equivalent to the nonvanishing of the obstruction class of $E_W$ (Theorem \ref{tangent-bundle-1}). An analogous result holds for the complexified adjoint representation. We prove that if $\hh$ is a complex Lie subalgebra of the complexified Lie algebra $\frg_\mathbb C$, then the following three are equivalent (Theorem \ref{complexified-1}):  (i) the top Chern class of the associated bundle $E_W$ is nonzero, (ii) $\hh$ contains a Borel subalgebra, and (iii) $\hh$ is universal in $\frg_\mathbb C$ under the complexified adjoint representation. This is in some sense a partial converse to Schur's triangularization.

Then we study all the irreducible representations of $SU(2)$, and consider all the $T$-invariant subspaces $V$ which satisfy the dimension requirement (\ref{dimension}). We give a complete description of these subspaces, show that they are all universal, and give a necessary and sufficient condition for the Euler class of $E_W$ to vanish (Proposition \ref{SU(2)}). A byproduct of this analysis is that we get an infinite family of counterexamples to Question \ref{question}.

Let $G$ be a connected solvable Lie group. We prove that if $\rho:G\to GL(U)$ is a complex linear representation, then the only complex linear subspace $V\subset U$ that is universal in $U$, is $U$ itself (Proposition \ref{solvable}).

Let $G$ be a compact connected Lie group, $\rho:G\to GL(U)$ a complex linear representation, and $V\subset U$ a complex linear universal subspace. We show that if $S\subset G$ is a Levi subgroup, then $V$ is universal for $S$ also (Proposition \ref{compact-levi}). But the same is no longer true if we take real representations for $G$ or if $G$ is not compact (Remark \ref{solv-remark}).

\section{Some important cases of universal subspaces}
In this section, we discuss three important classes of examples where the universality is equivalent to the nonvanishing of the obstruction class for a certain vector bundle.

Let $p:E\to M$ be a smooth real oriented vector bundle of rank $r$ over a base space $M$ which is a CW-complex. Then $E$ has the Euler class $e(E)\in H^r(M,\ZZ)$. If the rank $r$ is equal to the dimension of the base $M$, then $E$ admits a nowhere vanishing section if and only if the Euler class vanishes (VII Corollary 14.4 on p. 514 of Bredon \cite{B}). If $E$ is a complex vector bundle of rank $n$, then the top Chern class $c_n(E)\in H^{2n}(M,\ZZ)$ is the same as the Euler class of the underlying real vector bundle. 

\subsection{Holomorphic case}

Suppose $G$ is a complex Lie group (not necessarily compact). Let $U$ be a complex vector space and let $\rho:G\to GL(U)$ be a holomorphic representation. Let $V$ be a complex linear subspace of $U$, and let $G_V=\{g\in G|~\rho(g)(v)\in V,~\forall v\in V\}$. Then $G_V$ is a closed holomorphic Lie subgroup and the quotient $G/G_V$ is a holomorphic manifold. We assume now that $G/G_V$ is compact and the dimension requirement (\ref{dimension}) holds. The Lie subalgebra $\frg_V=\{X\in \frg:d\rho(X)(V)\subset V\}$ is the Lie algebra of the connected component of $G_V$.

\begin{theorem} \label{holomorphic-case-1} Let the notation be as above. Then the following are equivalent.
	
	\begin{enumerate}
		\item If $e(E_W)\in H^{2r}(G/G_V,\ZZ)$ is the Euler class, and if $[G/G_V]\in H_{2r}(G/G_V,\ZZ)$ is the fundamental class, then $C_V=\langle e(E_W),[G/G_V]\rangle \ne 0.$
		
		\item $V$ is universal.
	\end{enumerate}
\end{theorem}

\begin{proof}
	 Our argument follows the proof of Theorem 4.4 of \cite{A-D1}. We briefly recall some of the ingredients in their proof in our context. The Lie algebra map $d\rho:\frg\to \frg l(U)$ is complex linear. Let $W=U/V$, and consider the projection map $\pi_W:U\to W$. Given $v\in V$, the map \[\frg\to W: X(\in \frg)\mapsto -\pi_W(d\rho(X)v)\] is zero on the subspace $\frg_V$, so it indeces a complex linear map $\psi_v$ as in (4.1) of \cite{A-D1}:
	
	\[\psi_v:\frg/\frg_V\to W.\]

	However in our case, if $\psi_v|_\RR$ is the corresponding real linear map of the underlying real vector spaces, then its determinant is always nonnegative. Indeed, $det(\psi_v|_\RR)=|det(\psi_v)|^2\ge 0$. Therefore, for every $u\in U$, if $x\in G/G_V$ is a zero of the section $s_u$, then ${\rm ind}_x(s_u)={\rm sgn}(\psi_v)\ge 0$.
	
	We observe that the rest of the proof of Theorem 4.4 of \cite{A-D1} relies on three main assumptions: the orientability of $E_W$, the orientability and compactness of the base manifold, and the dimension requirement (\ref{dimension}). The same proof carries over to the situation where $E_W$ is complex, the base manifold $G/G_V$ is compact, holomorphic, and (\ref{dimension}) holds. 
\end{proof}

\subsection{Universal Lie subalgebras for the adjoint representation}

In this section, we examine Question \ref{question} for the Lie subalgebras in the adjoint representation of a compact connected Lie group. To be able to state our result regarding the obstruction class of a bundle on the corresponding homogeneous space,  we need to first observe that maximal rank subalgebras are always the Lie algebras of \emph{closed} Lie subgroups. This is an old result of Borel and de Siebenthal (cf. Theorem in \S 8 of \cite{B-S}). We give a different proof of this result based on our key observation (Proposition \ref{universal-closed}).

Let $G$ be a connected Lie group (not necessarily compact) and let $\frg$ be its associated Lie algebra.  Recall that if $\hh$ is a Lie subalgebra of $\frg$, then there is a connected Lie group $H$, an injective smooth group homomorphism $i:H\to G$ such that $di$ is injective and $di({\rm Lie} (H))=\hh$. This is sometimes referred to as a ``virtual Lie subgroup'', because the image $i(H)$ is not necessarily closed. For example, take the torus $G=U(1)\times U(1)$, and take $\hh$ to be the real line through the point $(1,\alpha)\in \RR^2={\rm Lie}(G)$, where $\alpha$ is irrational. It is well known that the image of the corresponding virtual Lie subgroup is dense in $G$. However, in Proposition \ref{universal-closed} we show that if the Lie subalgebra $\hh$ is \emph{universal} for the adjoint representation of $G$ then the Lie subgroup $H$ is necessarily closed. 

We denote by $N_G(H)$
 the normalizer of $H$ in $G$ and for a Lie subalgebra $\hh$ of $\frg={\rm Lie}(G)$, the normalizer of $\hh$ in $G$ is defined by $N_G(\hh):=\{g\in G|~ {\rm Ad}_g(\hh)\subset \hh\}$.

\begin{proposition}\label{universal-closed} Let $G$ be a connected Lie group and let $\hh\subset \frg$ be a universal subspace for the adjoint action of $G$. If $\hh$ is a Lie subalgebra of $\frg$, then $\hh$ is the Lie algebra of a \emph{closed} connected subgroup $H$ of $G$. Moreover, $H=N_G(\hh)^\circ=N_G(H)^\circ$. \end{proposition}

\begin{proof} Note that we have a surjection $G\times \hh\to \frg:(g,X)\mapsto {\rm Ad}_g(X)$. Let $G_\hh^\circ$ be the connected component of the normalizer $N_G(\hh)$. Then the above surjection factors through the surjection $E(\hh)\to \frg$, where $E(\hh)=(G\times \hh)/G_\hh^\circ$, the associated $G$-equivariant vector bundle on $G/G_\hh^\circ$. 
	Since this is a surjection, by Sard's theorem (cf. Theorem 6.2 \cite{B}), we can say that (all real dimensions) 
	\[\dim(G)-\dim(G_\hh^\circ)+\dim(\hh)\ge \dim(\frg)=\dim(G),\]
	\[{\rm or},\;\dim(\hh)\ge \dim(G_\hh^\circ).\]
 Since $\hh$ is a Lie subalgebra, there must be a connected virtual Lie subgroup $i:H\to G$ such that $di(Lie(H))=\hh$. But $i(H)\subset G_\hh^\circ$, which means $$\dim(G_\hh^\circ)\ge \dim(H)=\dim(\hh).$$ Therefore, $\dim(H)=\dim(G_\hh^\circ)$, which means $i:H\to G_\hh^\circ$ must be surjective, as $G_\hh^\circ$ is connected. We have a one-one, onto map $i$ between two connected manifolds, and it induces isomorphism at all tangent spaces. Therefore $i$ is to be a diffeomorphism by the inverse function theorem. Since $G_\hh$ is closed in $G$ by definition, so is its connected component $G_\hh^\circ$.

\end{proof}

As a corollary, we get a different proof for the following result of A. Borel and J. De Siebenthal (cf. Theorem in \S 8 of \cite{B-S}). Recall that a lie subalgebra $\hh$ of $\frg$ is said to be {\it maximal rank} if $\hh$ contains a maximal torus of $\frg$. 

\begin{corollary}\label{max-rank-1} Let $G$ be a compact Lie group and let $\mathfrak k\subset \frg$ be a Lie subalgebra of \emph{maximal rank}. Then there is a connected \emph{closed} Lie subgroup $K\subset G$ such that $\mathfrak k$ is the Lie algebra of $K$.\end{corollary}

\begin{proof}
	If $G$ is compact, then in $\frg$, every maximal rank subalgebra $\mathfrak k$ contains a maximal toral subalgebra. Since maximal toral subalgebras are universal in $\frg$, so is $\mathfrak k$. Therefore by Proposition \ref{universal-closed}, $\mathfrak k$ must be the Lie algebra of a \emph{closed} subgroup $K$.
      \end{proof}

\begin{theorem}\label{tangent-bundle-1}
  Let $G$ be a compact connected Lie group, $\hh\subset \frg$ any Lie subalgebra. Then the following are equivalent.
  	\begin{enumerate}
  	\item $\hh$ of maximal rank.
  	
  	\item In the adjoint representation, $\hh$ is universal in $\frg$.
  	
  	\item $\hh$ is the Lie algebra of a closed connected subgroup $H$, and the Euler characteristic of $G/H$ is positive.
  	
  	\item $\hh$ is the Lie algebra of the connected component of $G_\hh=N_G(\hh)$. If $W=\frg/\hh$, $E_W$ is the $G$-equivariant bundle on $G/G_\hh$ associated to $W$, then the obstruction class to having a nowhere vanishing section of $E_W$ on $G/G_\hh$ is nonzero.
  \end{enumerate}
  
\end{theorem}

\proof We will prove (1) $\Rightarrow$ (2) $\Rightarrow$ (1) $\Rightarrow$ (3) $\Rightarrow$ (4) $\Rightarrow$ (2).
  
$(1) \Rightarrow (2):$ Since $\hh$ is of maximal rank, it contains a maximal toral algebra. Any maximal toral algebra is universal for the adjoint representation.

$(2) \Rightarrow (1):$ Since $\hh$ is universal, it is the Lie algebra of a closed connected Lie subgroup $H$ by \ref{universal-closed}. Let $T$ be a maximal torus of $G$ and let $a$ be a generating element for $T$. For a Lie group $G$ and its Lie algebra $\frg$, let $\exp:\frg \to G$ be the exponential map. Let $a=\exp(A)$ for some $A\in \mathfrak t$. If $\hh$ is universal, then there exists $x\in G$ such that ${\rm Ad}_x A\in \hh$.  
Then $xax\inv=\exp({\rm Ad}_x A)\in \exp(\hh).$ Note that $H$ is a compact connected Lie group, hence $\exp(\hh)=H$. This implies that $x T x\inv\subset H$, in other words, $\hh$ is of maximal rank.

$(1) \Rightarrow (3):$ By (2), Corollary \ref{max-rank-1}, $\hh$ is the Lie algebra of a closed connected Lie subgroup $H$ of maximal rank. This is equivalent to saying that $H$ contains a maximal torus of $G$. In that case, a theorem of Hopf and Samelson \cite{H-S} says that the Euler characteristic of $G/H$ is positive.

 $(3) \Rightarrow (4):$
If $G/H$ is \emph{orientable} (which is the case always, if $H$ is connected by Proposition 15.10 on p. 92 of \cite{Bu})
 then by  Corollary 11.12 of \cite{M-S}, the Euler characteristic \[\chi(G/H)=\langle e(T(G/H)),[G/H]\rangle\] which means the Euler class of the tangent bundle of $G/H$ is nonzero.

Recall that $G_\hh=\{g\in G|~{\rm Ad}_gv\in \hh,~\forall v\in \hh\}=N_G(\hh)$. Since $H$ is connected, $N_G(H)= N_G(\hh)$. Now we make an intermediate lemma.

\begin{lemma} Let $G$ be a compact connected Lie group and let $T\subset G$ be a maximal torus. Let $H\subset G$ be a closed subgroup containing $T$. Then $N_G(H)/H$ is finite.\end{lemma}

\proof 
If $x\in N_G(H)$, then $x\inv Tx\subset x\inv Hx=H$. Since $T$ and $x\inv Tx$ are both maximal tori of $H$, there is $h\in H$ such that $x\inv Tx=hTh\inv$, which implies $xh\in N_G(T)\cap xH$. Since $T\subset H$, the entire coset $xhT\subset xH$. Since $N_G(T)/T$ is finite, there are only finitely many cosets $yT$ where $y\in N_G(T)$. Since $xH$ are either \emph{disjoint} or equal, this implies there are only finitely many cosets $xH$.\hfill$\Box$\\

We resume the proof of Theorem \ref{tangent-bundle-1}. Writing $G_\hh=N_G(\hh)$, $G_\hh/H=N_G(H)/H$ is finite. Therefore the projection $p:G/H\to G/G_\hh$ is a covering map, so the pullback of the tangent bundle is the tangent bundle. If the tangent bundle below has a nowhere vanishing section $\sigma$, the tangent bundle above admits the nowhere vanishing section $p^*\sigma$, which is impossible because the Euler class of $T(G/H)$ is nonzero by the above. Consider the representation of $G_\hh$ on $W=\frg/\hh$ induced from the adjoint representation of $G$ on $\frg$. 
The tangent bundle $T(G/G_\hh)$ is isomorphic to the associated bundle $E_W$ on $G/G_\hh$. Therefore we have proved that if $\chi(G/H)>0$, then the obstruction class of $E_W=T(G/G_\hh)$ for the existence of a nowhere vanishing section is nonzero.

$(4) \Rightarrow (2):$ As in \cite{A-D1}, for every $u\in \frg$, the constant section of $E_\frg\cong (G/G_\hh)\times \frg$ defines a section $s_u$ of the quotient bundle $E_\frg/E_\hh=E_W$. Since $s_u$ has a zero for every $u\in \frg$, $\hh$ is universal.

  \subsection{Universal Lie subalgebras for the complexified adjoint representation} 

  In this section, we prove an analogue of Theorem \ref{tangent-bundle-1} for the complexified adjoint representation of a compact connected Lie group.

  The complexification of a real Lie algebra $\frg$ is the complex Lie algebra $\frg_\mathbb C := \frg\otimes_{\mathbb R} \mathbb C $, where the complex Lie bracket is given by $$[X+iY, S+iT]=[X,S]-[Y,T]+i([Y,S]+[X,T])~ {\rm for}~X,~Y,~S,~T \in \frg.$$ If $\frg\subset \frg l(n, \mathbb R)$, then the complexification $\frg _\mathbb C$ can be identified with the Lie algebra $\frg + i\frg\subset \frg l(n, \mathbb C).$

  For every compact connected Lie group $G$, there is a unique reductive algebraic group $G_\CC$, called the \emph{complexification of $G$}, with the following properties. (i) $G$ is a maximal compact Lie subgroup of $G_\CC$. (ii) The natural map ${\rm Lie}(G)\otimes_\RR \CC\to {\rm Lie}(G_\CC)$ is an isomorphism, in other words, ${\rm Lie}(G_\CC)$ is the complexification of ${\rm Lie}(G)$. (iii) Consider the category $\mathscr C$ of finite dimensional unitary representations of $G$, and the category $\mathscr C'$ of rational algebraic complex representations of $G_\CC$. Then restriction to the subgroup $G$ induces an equivalence of categories $\mathscr C'\to \mathscr C$. In other words, every unitary representation $\rho:G\to GL(U)$ is the restriction of a unique rational representation $\rho_C:G_\CC\to GL(U)$. (iv) If $\rho:G\to GL(U)$ is a faithful unitary representation, then $\rho_C$ is an isomorphism of $G_\CC$ with the Zariski closure of $\rho(G)$ in $GL(U)$. For details, see Processi \cite{P}.

   Let us now recall that a Borel subgroup of a complex reductive algebraic group is a maximal Zariski closed and connected solvable algebraic subgroup. Such a subgroup always exists for dimension reason.
  Let $G$ be a compact connected Lie group and consider the complexified adjoint representation on $\frg_\mathbb C$. Let $G_\mathbb C$ be the complexification of $G$. If $B$ is a Borel subgroup of $G_\CC$ then  there is a maximal torus $T$ of $G$ such that $T=G\cap B$. Let $\bb={\rm Lie}(B)$. We have the following natural identifications, where the first one is a $G$-equivariant diffeomorphism and the second one is a $T$-equivariant isomorphism of real vector spaces (cf. Theorem 1, \S 7.3, page 382 \cite{P}).

  \begin{equation}\label{gmodt}G/T\isomor G_\mathbb C/B, \; \;\;\frg/\mathfrak t\isomor \frg_\mathbb C/\bb\end{equation}

 As a corollary, we recover the following well known result, which can also be seen as a special case of Theorem 5.7 of \cite{A-D1}. 
 
  \begin{corollary}\label{schur}\emph{(Schur's triangularization)} Let $G$ be a compact connected Lie group with Lie algebra $\frg$, and let $\bb\subset \frg_\CC$ be a Borel subalgebra. Then $\bb$ is universal in $\frg_\CC$ for $G$.\end{corollary}

  \proof Since $\mathfrak t$ is universal in $\frg$ for the adjoint representation of $G$, the Euler class of the quotient bundle $E_W=E_\frg/E_\mathfrak t$ on $G/T$ is nonzero, where $W=\frg/\mathfrak t$ (Theorem \ref{tangent-bundle-1}). On the other hand, for the complexified adjoint representation of $G$ on $\frg_\CC$ and the subspace $\bb$, $W'=\frg_\CC/\bb$, the quotient bundle $E_{W'}=E_{\frg_\CC}/E_\bb$ is isomorphic to $E_W$ by (\ref{gmodt}). Thus the Euler class of $E_{W'}$ is also nonzero, so $\bb$ is universal in $\frg_\CC$.\hfill$\Box$
  \vskip 1em

  The following is immediate from Proposition \ref{universal-closed}.

\begin{corollary}\label{universal-closed-complex} Let $G$ be a compact connected Lie group and let $\hh\subset \frg_\mathbb C$ be a universal complex subspace for the complexified adjoint action of $G$. If $\hh$ is a complex Lie subalgebra of $\frg$, then $\hh$ is the Lie algebra of a \emph{closed} connected complex Lie subgroup $H\subset G$. Moreover, $H=N_{G_\mathbb C}(\hh)^\circ=N_{G_\mathbb C}(H)^\circ$. \end{corollary}

  Recall that a parabolic subgroup $P$ of the complex reductive algebraic group $G_\CC$ is a connected closed subgroup in the Zariski topology, for which the quotient space $G_\CC/P$ is a projective complex algebraic variety. A subgroup $P$ of $G$ is a parabolic subgroup if and only if it contains some Borel subgroup of the group $G$ (cf. Corollary in \S 11.2 of \cite{Bo})

  \begin{corollary}\label{tangent-bundle-2} Let $G$ be a compact connected Lie group, $\hh$ any complex Lie subalgebra of $\frg_\mathbb C$ containing a Borel subalgebra. Let $W=\frg_\CC/\hh$, $G_\hh=N_G(\hh)$. Then
\begin{enumerate}
	\item ${\rm dim}(G/G_\hh)={\rm rank}_\mathbb R(W|_\mathbb R)$.
		
\item The top Chern class of the associated bundle $E_W$ on $G/G_\hh$ is also nonzero.
  \end{enumerate}
      \end{corollary}
\proof Suppose $\hh$ contains a Borel subalgebra $\bb$. By Schur's triangularization above, $\hh$ is universal for the complexified adjoint representation of $G$, hence also for the adjoint representation of $G_\CC$. By Corollary \ref{universal-closed-complex}, $\hh$ is the Lie algebra of a \emph{closed} connected complex analytic subgroup of $H\subset G_\mathbb C$, with $H=(G_\CC)_\hh^\circ$. Putting $\hh_\mathbb R=\frg\cap \hh$, $H_\mathbb R=G\cap H$, we have $\hh_\mathbb R={\rm Lie}(H_\mathbb R^\circ)$. Since the adjoint 
 representation $G_\mathbb C\to \frg l(\frg_\mathbb C)$ is algebraic, the stabilizer $(G_\mathbb C)_\hh$ is an \emph{algebraic subgroup}. But then the identity component
  in the Zariski topology will be connected in Euclidean topology (\cite{S} VII. 2.2 Theorem 1), 
  so $H=(G_\mathbb C)_\hh^\circ$ must be an algebraic subgroup of $G_\mathbb C$. Therefore $H$ is a parabolic subgroup of $G_\CC$, hence $H=N_{G_\CC}(H)$ (Theorem 11.16 of Borel \cite{Bo}). But $H$ is connected, so $(G_\CC)_\hh=N_{G_\CC}(H)=H$. This means $H_\RR=G\cap H=G\cap (G_\CC)_\hh=G_\hh$. Now since $\frg\to \frg_\mathbb C/\bb$ is surjective, $\frg\to \frg_\mathbb C/\hh$ is also surjective, and hence we get the following isomorphism of real vector spaces, equivariant under $H_\mathbb R$.\begin{equation}\label{gmodh-1}\frg/\hh_\mathbb R\isomor \frg_\mathbb C/\hh.\end{equation}
 We have already shown $H_\RR=G_\hh$, so this proves $(1)$.

Also, since $G$ acts transitively on $G_\mathbb C/B$, $G$ must act transitively on the further quotient $G_\mathbb C/H$, the stabilizer of $[H]\in G_\CC/H$ being $H_\mathbb R$. Thus, we have a $G$-equivariant isomorphism \begin{equation}\label{gmodh}G/H_\mathbb R\isomor G_\mathbb C/H.\end{equation}

Now since $H$ contains $B$, $H_\mathbb R$ contains $T$. If $W=\frg/\hh_\mathbb R$ with the induced action of $H_\mathbb R$, then $E_W\cong T(G/H_\mathbb R)$. If $W'=\frg_\CC/\hh$ with the induced action of $H$, then $E_{W'}\cong T(G_\mathbb C/H)$. Under (\ref{gmodh}) and (\ref{gmodh-1}), the bundle $E_W$ is isomorphic to $E_{W'}$.  The Euler class of this bundle must not vanish by Theorem \ref{tangent-bundle-1} or Theorem \ref{holomorphic-case-1}.\hfill$\Box$

\vskip 1em

 \begin{theorem}\label{complexified-1} Let $G$ be compact connected Lie group and let $G_\mathbb C$ be its complexification. For a complex Lie subalgebra $\hh\subset \frg_\mathbb C$, the following are equivalent.
 	
 	\begin{enumerate}
 		\item $\hh$ is universal in $\frg_\CC$ for the complexified adjoint action of $G$.
 		
 		\item $\hh$ contains a Borel subalgebra of $\frg_\CC$.

	\item $\hh$ is the Lie algebra of a parabolic subgroup $H\subset G_\mathbb C$ and the natural map $G/(G\cap H)\to G_\mathbb C/H$ is a diffeomorphism.
	
	\item The associated vector bundle $E_W$ on $G/G_\hh$, where $W=\frg_\CC/\hh$, has nonvanishing top Chern class.
\end{enumerate}
\end{theorem}

Before we prove Theorem \ref{complexified-1} we mention the following lemma, which is obvious but useful.

\begin{lemma}\label{immersion} Let $M$ be a smooth connected manifold of dimension $n$ and let $i:K\to M$ be an injective immersion, where $K$ is a compact connected manifold of dimension $n$. Then $i$ is a diffeomorphism.\end{lemma}

\vskip 2em
\noindent\emph{Proof of Theorem \ref{complexified-1}}:
$(4)\Rightarrow (1)$ is obvious. 

$(3)\Rightarrow (4):$ In this case $\hh$ is parabolic. Then the implication follows from Corollary \ref{tangent-bundle-2}.

$(1)\Rightarrow (3):$
Since
 $\hh$ is universal for $G$, 
 the map $G\times \hh\to \frg_\mathbb C$ is surjective, which means the 
 map $E_\hh\to \frg_\mathbb C$ is surjective, where $E_\hh=(G\times\hh)/G^\circ_\hh$.
  Now, since $\hh$ must also be universal for $G_\mathbb C$, by
  Corollary \ref{universal-closed-complex} (see also Proposition \ref{universal-closed}), $\hh$ is the Lie algebra of a closed complex Lie subgroup $H$ such that $H=N_{G_\mathbb C}(H)^\circ$. Since the adjoint 
 representation $G_\mathbb C\to \mathfrak{gl}(\frg_\mathbb C)$ is algebraic, the stabilizer $(G_\mathbb C)_\hh$ is an \emph{algebraic subgroup}. But then the identity component
  in the Zariski topology will be connected in Euclidean topology (\cite{S} VII. 2.2 Theorem 1), 
 so $H=(G_\mathbb C)_\hh^\circ$ must be an algebraic subgroup of $G_\mathbb C$. 
  Again, $G_\hh=G\cap (G_\mathbb C)_\hh=G\cap N_{G_\mathbb C}(H)$, so we have
  $$G_\hh^\circ= (G\cap N_{G_\mathbb C}(H)^\circ)^\circ=(G\cap H)^\circ.$$
   But by surjectivity of $E_\hh\to \frg_\mathbb C$, we know that
	\[\dim_\RR(G)-\dim_\RR(G\cap H)^\circ+\dim_\RR(\hh)=\dim_\RR(E_\hh)\ge \dim_\RR(G_\mathbb C),\]
\[{\rm or},\;\dim_\RR(G/G\cap H)\ge \dim_\RR(G_\mathbb C/H).\]

 Since $G/G\cap H$ is compact and the map $G/G\cap H\to G_\mathbb C/H$ is an injective immersion, they must have the same dimension, so by Lemma \ref{immersion}, $G/G\cap H\to G_\mathbb C/H$ is a diffeomorphism, proving that $G_\mathbb C/H$ is compact. Since $H$ is an algebraic subgroup of $G_\mathbb C$, the quotient space $G_\mathbb C/H$ is quasi-projective, or in other words, it is open in an irreducible projective variety $M$. By \cite{S} VII. 2.2 Theorem 1, $M$ is connected in Euclidean topology. Since $G_\mathbb C/H$ is compact, it is equal to $M$. This shows that $H$ is parabolic. 

$(3)\Rightarrow(2):$ Follows from the definition of a parabolic subgroup. 

$(2)\Rightarrow(1):$ This follows from the general Schur's triangularization (Corollary \ref{schur}).
 \hfill$\Box$

\section{Universality does not imply nonzero Euler class}\label{counter}

First we make preparatory observations. Consider the complexified adjoint representation of $G=SU(2)$ on $U=\frg_\mathbb C=\mathfrak{su}(2)\otimes_\RR \CC\cong \mathfrak{sl}(2,\CC)$. Let $V=\bb$ be the standard Borel subalgebra of $\frg_\mathbb C$, that is, $\bb$ consists of upper triangular traceless matrices in $\mathfrak{sl}(2,\CC)$. Let $W=U/V$. Then $G_V=T$, the standard diagonal (maximal) torus in $SU(2)$, $G/G_V\cong S^2$, $E_W\cong T(S^2)$. Indeed, for the adjoint representation of $G_\mathbb C=SL(2,\CC)$ on its Lie algebra $U$, $(G_{\mathbb C})_V=B$, and $G/G_V\to G_\mathbb C/(G_{\mathbb C})_V$ is a $G$-equivariant isomorphism, while $E_W$ as a bundle on $G_\mathbb C/(G_{\mathbb C})_V$ is isomorphic to its tangent bundle.

By Corollary 14.4 of \cite{B}, if $E$ is an orientable vector bundle on a compact manifold $M$, ${\rm rank}(E)=\dim(M)$, then the Euler class $e(E)$ is the final obstruction class to having a nowhere vanishing section of $E$. In case $E$ is the underlying real vector bundle of a complex vector bundle $E'$, the Euler class $e(E)$ is equal to the top Chern class $c_{\rm top}(E')$.

In what follows, we construct an example of a compact Lie group $G$, a complex representation $U$, a universal complex linear subspace $V$, such that the dimension condition (\ref{dimension}) holds, but the Euler class of $E_W$ over $G/G_V$ is zero.

Let $G_1=G_2=SU(2)$ and $G=G_1\times G_2$. Let $U_1=U_2=\mathfrak{sl}(2,\CC)$, and consider $U=U_1\oplus U_2={\rm Lie}(G)_\CC$ with the complexified adjoint action of $G$. Let $V_2=\left\{\begin{pmatrix} 0 & *\\ * & 0\end{pmatrix}\right\}$ be the set of zero diagonal matrices in $U_2$, and let $V_1=\bb$ be the standard Borel in $U_1$. Set $V=V_1\oplus V_2\subset U$, $W_i=U_i/V_i$ for $i=1,2$, $W=U/V=W_1\oplus W_2$. Then $G_V=T\times N_G(T)$, and $V$ is universal in $U$. Now, $G/G_V\cong S^2\times \RR P^2$. We will denote the two projections by $p_1,p_2$ respectively. Consider the vector bundles $E_1=E_{W_1}$ on $S^2$ and $E_2=E_{W_2}$ on $\RR P^2$. Then the mod $2$ reduced Chern class of $E_2$ is the generator of $H^2(\RR P^2,\ZZ/2)=\ZZ/2=H^2(\RR P^2,\ZZ)$
as observed in the counterexample of Jinpeng and Dokovi\'{c} (\S 6, \cite{A-D1}).

 On the other hand, since $E_1$ is the tangent bundle on $S^2$, its Chern class is twice the generator of $H^2(S^2,\ZZ)$. 
We will show below that the top Chern class of the bundle $E_W=p_1^*E_1\oplus p_2^*E_2$ is zero. By the Whitney sum formula, $c_2(E)=p_1^*c_1(E_1)\cup p_2^*c_1(E_2)$, which is the image of $c_1(E_1)\otimes c_1(E_2)$ under the cross product map
$$H^2(S^2,\ZZ)\otimes H^2(\RR P^2,\ZZ)\to H^4(G/G_V,\ZZ): (a,b)\mapsto p_1^*a\cup p_2^*b.$$ But  $H^2(S^2,\ZZ)\otimes H^2(\RR P^2,\ZZ)\cong\ZZ\otimes_\ZZ \ZZ/2\ZZ\cong \ZZ/2\ZZ$. Since $c_1(E_1)$ is twice the generator of $H^2(S^2,\ZZ)$, $c_2(E)=0$.

 \section{Universality of irreducible representations of $SU(2)$}
Let $T$ be the standard torus in $SU(2)$. For any irreducible representations of $SU(2)$, we consider all the $T$-invariant complex linear subspaces $V$ which satisfy the dimension requirement (\ref{dimension}), and show that they are universal. We give a concrete description of these subspaces give a necessary and sufficient condition for the Euler class of $E_W$ to vanish.

 \begin{proposition} \label{SU(2)}
 	Let $G=SU(2)$, $T\subset G$ a fixed maximal torus, and let $\rho_n:G\to GL(U_n)$
 	 be an irreducible unitary finite dimensional complex representation of $G$ with ${\rm rank}_\CC(U_n)=n+1$.
 	
 	\begin{enumerate}
 		\item The only $T$-invariant complex subspaces $V\subset U_n$ such that the dimension requirement (\ref{dimension}) holds are hyperplanes, and they are universal.
 		
 		\item There are exactly $n+1$ many choices for $V$ (and their complements $W$). We can give concrete descriptions of them.
 		
 		\item The Chern class of $E_W$ over $G/G_V$ vanishes if and only if $n=4m$ and $T$ acts trivially on $W$.
 	\end{enumerate}

   \end{proposition}

 \begin{proof}
 	  We know that $U_n$ is equivalent to the standard representation on the space of complex homogeneous polynomials in $x,y$ of degree $n$. If $V$ is a subspace such that $G_V$ contains the maximal torus, then $G_V$ is of dimension either $1$ or $3$, because the only possible subgroups of $SU(2)$ are of dimension $0,1,3$. Then the only nontrivial case is when $G_V$ is of dimension $1$, which means $T$ is the connected component of $G_V$ i.e. $G_V=T$ or $N_G(T)$ (the connected component has to be a normal subgroup). Fix the Weyl element $w:x\mapsto y\mapsto -x$. Now since the complex subspace $W=V^\perp$ has real rank equal to the dimension of $G/G_V$, which is $2$, $W$ has to be a complex line. This implies $V$ is a hyperplane. Therefore $V$ is universal in $U_n$ by Gichev \cite{G} Corollary 1. This proves $(1)$.

          The only possible $T$-invariant lines are the eigenspaces of $T$ in $U_n$ for different characters. They are precisely the complex lines containing the monomials $x^iy^{n-i}$, hence they are $n+1$ many. Let $W$ be the complex line through the monomial $x^iy^{n-i}$. Thus we can concretely describe $V$ as the space of polynomials $f\in U_n$ such that the coefficient of $x^iy^{n-i}$ in $f$ is zero. This proves $(2)$.

 The maximal torus $T$ acts on the line through the monomial $x^iy^{n-i}$ by the character $t\mapsto t^{2i-n}$. The stabilizer is the set of 
 $g=\begin{pmatrix}\alpha &\beta \\ -\overline{\beta},&\overline{\alpha}\end{pmatrix}\in SU(2)$ 
 such that $gx^iy^{n-i}$ lies in $W$. But this stabilizer is either $T$ or $N_G(T)$, as seen above. The Weyl element $w$ acts as $w\cdot x^iy^{n-i}=(-1)^{n-i}x^{n-i}y^i$, so it does not stabilize $W$ unless $i=n-i$. In case $2i\ne n$, the stabilizer $G_V=T$, $G/G_V\cong S^2$, and it is easy to show that the Chern class of the complex line bundle associated to this character $t\mapsto t^{2i-n}$ is $(2i-n)$ times the generator of $H^2(S,\ZZ)$, which is nonzero. 

 Now come to the case $n=2i$, where $T$ acts trivially on $W$. Here the Weyl element $w$ acts on $W$ trivially or nontrivially according as $i$ is even or odd. In case $w$ acts nontrivially, $G_V=N_G(T)$ and the associated complex line bundle $E_W$ on $G/G_V\cong \RR\PP^2$ is the complexification of the real line bundle which is the tautological real line bundle on $\RR\PP^2$. Since $H^*(\RR\PP^2,\mathbb F_2)=\mathbb F_2[w_1]/w_1^3$, where $w_1$ is the Stiefel-Whitney class of the tautological bundle, the mod 2 reduced Chern class of the complexification would be $w_1^2\ne 0$.  So the Chern class of $E_W$ itself is nonzero.  When $i=2m$ is even, the Weyl element acts trivially on $W$, so $E_W$ descends to the trivial complex line bundle on $G/N_G(T)=\RR P^2$. Therefore the Chern class of $E_W$ is zero. This completes the proof of $(3)$.

\end{proof}

 \begin{remark} By $(3)$ of Proposition \ref{SU(2)}, we get an infinite family of irreducible representations $U_{4m}$ of $G=SU(2)$, $T$-invariant universal subspaces $V_{4m}$ satisfying the dimension requirement (\ref{dimension}), such that the Chern class of $E_{W_{4m}}$ is zero. Here $W_{4m}$ can be concretely identified with the complex line through the monomial $x^{2m}y^{2m}$. This gives an infinite family of counterexamples to the original question posed in \cite{A-D1}. In contrust, the counterexample in \S\ref{counter} was worked out from scratch and did not depend on Gichev's result.
   \end{remark}

\section{Levi subgroups and solvable groups}

\begin{proposition}\label{solvable} Let $G$ be a connected solvable Lie group, $U=\CC^n$, and let $\rho:G\to GL(n,\CC)$ be a representation of $G$ on $U$. Then any complex linear universal subspace $V$ of $U$ has to be equal to $U$.\end{proposition}

\begin{proof} We prove it in two steps. In Case 1 we prove the result for simply connected solvable complex Lie groups, and in Case 2 for the general case.

{\bf Case 1.} $G$ is a complex simply connected solvable Lie group. Induction on the complex rank of $U$. If rank is $0$ there is nothing to prove. Now assume the rank is positive. Consider the action of the Lie algebra $d\rho:\frg\to\mathfrak{gl}(U)=\mathfrak{gl}(n,\CC)$. By Lie's theorem, $U$ admits a filtration $$U=U_n\supset U_{n-1}\supset \cdots \supset U_1\supset U_0=0$$ of complex subspaces invariant under $\frg$ (hence under $G$, because the connected group $G$ is generated by a neighbouhood of identity and such nbd is in the image of exp) such that $\frg$ acts on $U_i/U_{i-1}$ by characters. These characters integrate to characters on the simply connected $G$. Since the rank of $U$ is positive, $U_1\ne 0$. Since $V$ is universal and $G$ acts on $U_1$ by a character, $U_1\subset V$. Then $V/U_1$ is universal in $U/U_1$, and the induction hypothesis applies, so $V/U_1=U/U_1$.

{\bf Case 2.} The general case. First we claim that if $G$ is any connected solvable Lie group (real or complex), then there is a simply connected solvable complex Lie group $G_{sc}$ such that ${\rm Lie}(G_{sc})=\frg_\CC$. By Ado's theorem, there is an embedding of complex Lie algebras $\frg_\mathbb C\subset \mathfrak{gl}(N,\CC)$, so there is some connected complex Lie group $G'$ with $\frg_\mathbb C$ as the Lie algebra by the subgroup-subalgebra correspondence. This $G'$ has to be solvable if $\frg$ is. Now let $G_{sc}$ be the simply connected cover of $G'$. Since the central quotient $G'$ of $G_{sc}$ is solvable, $G_{sc}$ is solvable too. This proves the claim.

Given the representation $\rho$, consider $d\rho:\frg\to \mathfrak{gl}(U)$, and then $d\rho_\mathbb C=d\rho\otimes_\RR 1_\CC:\frg_\mathbb C\to \mathfrak{gl}(U)$. Then $d\rho_\mathbb C$ integrates to a representation $\rho_{sc}:G_{sc}\to GL(U)=GL(n,\CC)$. Now there is a group homomorphism $i:\tilde G\to G_{sc}$, where $\tilde G$ is the simply connected cover of $G$, and $di$ is the inclusion of $\frg$ into $\frg_\CC$. The two representations of $\tilde G$ on $U$ given by the following two compositions are equivalent.
$$\tilde G\to G\stackrel\rho\to GL(n,\CC)~~{\rm and}~~\tilde G\stackrel i\to G_{sc}\stackrel{\rho_\mathbb C}\to GL(n,\CC)$$ 
Since $V$ is universal for $G$, it is universal for $\tilde G$ and therefore for $G_{sc}$ as well. Now Case 1 applies. \end{proof}

Recall that a Levi decomposition of a connected Lie group $G$ is given by a virtual connected semisimple Lie subgroup $S\subset G$ such that $G=R\cdot S$, $R\cap S$ is discrete, where $R$ is the solvable radical of $G$. 
\begin{proposition}\label{compact-levi} Let $G$ be a compact connected Lie group, and let $G=R\cdot S$ be a Levi decomposition, where $R$ is the radical and $S$ is semisimple. Suppose $\rho:G\to GL(U)$ is a complex representation of $G$ and $V$ is a complex subspace that is universal for $G$. Then $V$ is universal for $S$ also.\end{proposition}

\begin{proof} 
  Since $G$ is compact, $R$ is central in $G$. Let $\Phi$ be the set of weights in the weight space decomposition of $U$ for the torus $R$. Then we have
	$$U=\oplus_{\alpha\in \Phi}U_\alpha,~{\rm where}~U_\alpha:=\{u\in U|~ \rho(g) u=\alpha(g)u,~\forall~g\in R\}.$$
	   Since $R$ is central in $G$, each $U_\alpha$ is $G$-invariant. Therefore, for each $\alpha\in \Phi$, $V\cap U_\alpha$ is universal in $U_\alpha$ for $G$.
	    Now our proposition will follow if we can show that $V\cap U_\alpha$ are universal in $U_\alpha$ for $S$. Let $v\in U_\alpha$. Then there is some $g\in G$ such that $\rho(g)v\in V\cap U_\alpha$. Write $g=sr$, $r\in R$, $s\in S$. Then $V\cap U_\alpha\ni \rho(g)v= \rho(s)\rho(r)v=\alpha(r)\rho(s)v$, which means $\rho(s)v\in V\cap U_\alpha$.\end{proof}

\begin{remark} \label{solv-remark}
	$(a)$ Propositions \ref{solvable} and \ref{compact-levi} are not always true for real representations. For example, take $G=U(1)$ acting on $U=\RR^2$ by rotation. Then any line through the origin is universal for $G$.
	
	$(b)$ Proposition \ref{compact-levi} is not valid if $G$ is not compact. We give an example to demonstrate this. Take \[G=\left\{\begin{pmatrix} A & B\\ 0 & A\end{pmatrix}\in SL(4,\CC): A\in SU(2), B\in \mathfrak{gl}(2,\CC).\right\}\]There is an isomorphism \[f:SU(2)\ltimes \mathfrak{gl}(2,\CC)\isomor G: (A,B)\mapsto \begin{pmatrix} A & AB\\ 0 & A\end{pmatrix}.\] Here the semidirect product is given by $(A',B')\cdot (A,B)=(A'A, A\inv B'A+B)$. 
	
	The subgroup $SU(2)$ sitting diagonally is a Levi component $S$, while the solvable radical is $R=\left\{\begin{pmatrix}I_2 & *\\ 0 & I_2\end{pmatrix}\right\}\cong \mathfrak{gl}(2,\CC)$, which is abelian. Let $U=\CC^4$ and consider the restriction of the defining representation of $SL(4,\CC)$ on $U$ to $G$. Let $V$ be the set of vectors $\{(v,w)\in \CC^4: v\in \CC e_1,w\in \CC e_1\}$, where $e_1=(1,0)\in \CC^2$. Then $V$ is universal for $G$. Indeed, if we take any $(v,w)\in \CC^4$, with $w\ne 0$, then we have some $A\in SU(2)$, $z\in \CC$ such that $Aw=ze_1$. Fix such an $A$. We claim that there is some $B\in\mathfrak{gl}(2,\CC)$, $z'\in\CC$ such that $Av+Bw=z'e_1$. If $Av=\alpha e_1$ for some $\alpha\in \CC$, we can find $B\in SU(2)\subset \mathfrak{gl}(2,\CC)$, $\beta\in \CC$ such that $Bw=\beta e_1$, so that $Av+Bw=(\alpha+\beta)e_1$. If $Av\not\in \CC e_1$, then $\{Av,e_1\}$ is a basis of $\CC^2$. Since $w\ne 0$, there is some $B_1\in GL(2,\CC)\subset \mathfrak{gl}(2,\CC)$ such that $B_1w=-\gamma Av+\delta e_1$, where $\gamma,\delta\in \CC$, $\gamma\ne 0$. If $B=\gamma\inv B_1$, then $Bw+Av=\delta\gamma\inv e_1$, which proves the claim. So, \[\begin{pmatrix} A & B\\ 0 & A\end{pmatrix}\begin{pmatrix} v\\ w\end{pmatrix}=\begin{pmatrix} z'e_1\\ ze_1\end{pmatrix}.\] If $w=0$, $v\ne 0$, then there are $A\in SU(2)$, $z\in \CC$ such that $Av=ze_1$. Thus \[\begin{pmatrix} A & 0\\ 0 & A\end{pmatrix}\begin{pmatrix} v\\ 0\end{pmatrix}=\begin{pmatrix} ze_1\\ 0\end{pmatrix}.\]
	Now $V$ is not universal for the Levi part $S$. For example, if $v=e_1$ and $w=e_2=(0,1)$ then there is no $A\in SU(2)$ such that $Av,Aw\in \CC e_1$. 
	\end{remark}

\vskip 2em
\parindent=0pt

{\small Saurav Bhaumik 
	
		Department of Mathematics, 
		
	Indian Institute of Technology Bombay, 
	
	Powai, Mumbai 400076, India
	
	\texttt{saurav@math.iitb.ac.in}
	\vskip 1em
	
	Arunava Mandal
	
	Department of Mathematical Sciences, 
	
	Indian Institute of Science Education and Research Mohali, 
	
	Punjab 140306, India
	
	\texttt{a.arunavamandal@gmail.com}

\end{document}